\documentclass[reqno, 10pt,a4]{amsart}

\usepackage{amscd,amsmath,amssymb,amsthm,amsfonts,epsfig,graphics}

\theoremstyle{plain}
\newtheorem{thm}{Theorem}[section]
\newtheorem{prop}[thm]{Proposition}

\newtheorem{lem}[thm]{Lemma}
\theoremstyle{definition}

\makeatletter
\def\mapstofill@{%
   \arrowfill@{\mapstochar\relbar}\relbar\longrightarrow}
\newcommand*\xmapsto[2][]{%
   \ext@arrow 0395\mapstofill@{#1}{#2}}
\makeatother
\numberwithin{equation}{section}

\title{Odd number and trapezoidal number 
}
\author{Masanori Ando}
\date{}
\address{Wakhok University, Hokkaido 097-0013, Japan}
\email{m-ando@wakhok.ac.jp}
\begin{document}
\pagestyle{empty}
\maketitle
\thispagestyle{empty}
\begin{abstract}
In this paper, we give a bijective proof of the reduced lecture hall partition theorem. 
It is possible to extend this bijection in lecture hall partition theorem. 
And refined versions of each theorems are also presented. 
\end{abstract}
\section{Introduction}
\[
\frac{1}{1}=1, \ \frac{3\cdot 2}{1\cdot 3}=2, \ \frac{6\cdot 5\cdot 3}{1\cdot 3\cdot 5}=6, 
\ \frac{10\cdot 9\cdot 7\cdot 4}{1\cdot 3\cdot 5\cdot 7}=24, \ \frac{15\cdot 14\cdot 12\cdot 9\cdot 5}{1\cdot 3\cdot 5\cdot 7\cdot 9}=120. 
\]
\indent 
There is the product of odd numbers in denominator. 
The numbers in numerator are the differences between two triangular numbers. 
And the calculation results are the factorial numbers. 
The proof of this result is not so difficult. 
We consider what generating function is the $q$-analogue of this calculation. 
\section{Odd number and trapezoidal number }
\begin{thm}\label{thm;1}
For any positive integer $n$, 
\[
\displaystyle 
\prod_{k=1}^{n}{\frac{\binom{n+1}{2}-\binom{k}{2}}{2k-1}}
=n!. 
\]
\end{thm}
\noindent
\textbf{Figure. }\\
\unitlength 0.1in
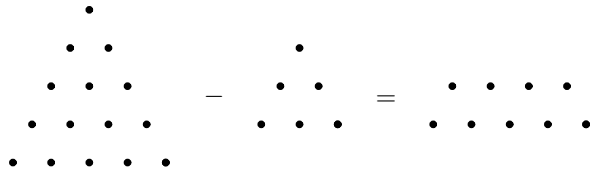
\begin{picture}(24.00,8.00)(0.00,-12.00)
\special{pn 20}%
\special{sh 1}%
\special{ar 1000 400 10 10 0  6.28318530717959E+0000}%
\special{sh 1}%
\special{ar 900 600 10 10 0  6.28318530717959E+0000}%
\special{sh 1}%
\special{ar 1100 600 10 10 0  6.28318530717959E+0000}%
\special{sh 1}%
\special{ar 800 800 10 10 0  6.28318530717959E+0000}%
\special{sh 1}%
\special{ar 1000 800 10 10 0  6.28318530717959E+0000}%
\special{sh 1}%
\special{ar 1200 800 10 10 0  6.28318530717959E+0000}%
\special{sh 1}%
\special{ar 700 1000 10 10 0  6.28318530717959E+0000}%
\special{sh 1}%
\special{ar 900 1000 10 10 0  6.28318530717959E+0000}%
\special{sh 1}%
\special{ar 1100 1000 10 10 0  6.28318530717959E+0000}%
\special{sh 1}%
\special{ar 1300 1000 10 10 0  6.28318530717959E+0000}%
\special{sh 1}%
\special{ar 600 1200 10 10 0  6.28318530717959E+0000}%
\special{sh 1}%
\special{ar 800 1200 10 10 0  6.28318530717959E+0000}%
\special{sh 1}%
\special{ar 1000 1200 10 10 0  6.28318530717959E+0000}%
\special{sh 1}%
\special{ar 1200 1200 10 10 0  6.28318530717959E+0000}%
\special{sh 1}%
\special{ar 1400 1200 10 10 0  6.28318530717959E+0000}%
\special{sh 1}%
\special{ar 1400 1200 10 10 0  6.28318530717959E+0000}%

\put(16.0000,-9.0000){\makebox(0,0)[lb]{$-$}}%

\special{pn 20}%
\special{sh 1}%
\special{ar 2100 600 10 10 0  6.28318530717959E+0000}%
\special{sh 1}%
\special{ar 2000 800 10 10 0  6.28318530717959E+0000}%
\special{sh 1}%
\special{ar 2200 800 10 10 0  6.28318530717959E+0000}%
\special{sh 1}%
\special{ar 1900 1000 10 10 0  6.28318530717959E+0000}%
\special{sh 1}%
\special{ar 2100 1000 10 10 0  6.28318530717959E+0000}%
\special{sh 1}%
\special{ar 2300 1000 10 10 0  6.28318530717959E+0000}%
\special{sh 1}%
\special{ar 2300 1000 10 10 0  6.28318530717959E+0000}%

\put(25.0000,-9.0000){\makebox(0,0)[lb]{$=$}}%

\special{pn 20}%
\special{sh 1}%
\special{ar 2900 800 10 10 0  6.28318530717959E+0000}%
\special{sh 1}%
\special{ar 3100 800 10 10 0  6.28318530717959E+0000}%
\special{sh 1}%
\special{ar 3300 800 10 10 0  6.28318530717959E+0000}%
\special{sh 1}%
\special{ar 3500 800 10 10 0  6.28318530717959E+0000}%
\special{sh 1}%
\special{ar 2800 1000 10 10 0  6.28318530717959E+0000}%
\special{sh 1}%
\special{ar 3000 1000 10 10 0  6.28318530717959E+0000}%
\special{sh 1}%
\special{ar 3200 1000 10 10 0  6.28318530717959E+0000}%
\special{sh 1}%
\special{ar 3400 1000 10 10 0  6.28318530717959E+0000}%
\special{sh 1}%
\special{ar 3600 1000 10 10 0  6.28318530717959E+0000}%
\special{sh 1}%
\special{ar 3600 1000 10 10 0  6.28318530717959E+0000}%
\end{picture}%
\\
\\
\indent
Number $n$ means bottom length of trapezoid as it's showed in figure. 
It's natural to classify it by top length. 
However we won't do that.

For positive integer $k,n$, we define trapezoidal number 
\begin{eqnarray*}
{\Diamond}_{n,k}:=\left \{ 
\begin{array}{ll}
n+(n-1)+\cdots +(n-2k+2) &\ \ ({}^\forall k\leq \frac{n}{2})\\
n+(n-1)+\cdots +(2k-n-1) &\ \ ({}^\forall k> \frac{n}{2})
\end{array}
\right.
\end{eqnarray*}
\begin{lem}\label{lem;1}
For positive integer $k,n, k\leq n$, 
\[
\frac{{\Diamond}_{n,k}}{2k-1}=n-k+1. 
\]
\end{lem}
\begin{proof}
Let calculate the sum of the number sequence. 
\begin{eqnarray*}
{}^\forall k\leq \frac{n}{2}, &&\\
{\Diamond}_{n,k}&=&n+(n-1)+\cdots +(n-2k+2)\\
&=&\frac{1}{2}(n+(n-2k+2))\times (n-(n-2k+2)+1)\\
&=&\frac{1}{2}(2n-2k+2)\times (2k-1)\\
&=&(n-k+1)\times (2k-1), 
\end{eqnarray*}
\begin{eqnarray*}
{}^\forall k> \frac{n}{2}, &&\\
{\Diamond}_{n,k}&=&n+(n-1)+\cdots +(2k-n-1)\\
&=&\frac{1}{2}(n+(2k-n-1))\times (n-(2k-n-1)+1)\\
&=&\frac{1}{2}(2k-1)\times (2n-2k+2)\\
&=&(2k-1)\times (n-k+1). 
\end{eqnarray*}
\end{proof}
We permute numerator numbers by a ``Skip" like next example. 
Then Theorem is obvious. \\
\textbf{Example. }For $n=6$, 
\[
\frac{21\cdot 20\cdot 18\cdot 15\cdot 11\cdot 6}{1\cdot 3\cdot 5\cdot 7\cdot 9\cdot 11}=\frac{6\cdot 15\cdot 20\cdot 21\cdot 18\cdot 11}{1\cdot 3\cdot 5\cdot 7\cdot 9\cdot 11}
=6\cdot 5\cdot 4\cdot 3\cdot 2\cdot 1=6!. 
\]

\bigskip
We consider the $q$-analogue of theorem 2.1. 
By the previous reduction, 
\[
\prod_{k=1}^{n}{\frac{1-q^{\binom{n}{2}-\binom{k-1}{2}}}{1-q^{2k-1}}}
=\prod_{k=1}^{n}\sum_{i=0}^{n-k}{q^{i(2k-1)}}. 
\]
This is the generating function of the odd partitions which that 
the number of $2k-1$ part is less than or equal to $n-k$.

\section{Reduced lecture hall partition theorem}
Let $n$ be a positive integer. A partition $\lambda $ of $n$ is an integer sequence 
\[
\lambda =(\lambda_1,\lambda_2,\ldots,\lambda_\ell) 
\]
satisfying $\lambda _1 \geq \lambda _2 \geq \ldots \geq \lambda _\ell >0$ and $\displaystyle \sum _{i=1}^{\ell }{\lambda _i} =n$. 
We call $\ell (\lambda):= \ell$ the length of $\lambda $, 
and each $\lambda _i$ a part of $\lambda $. 
We let ${\mathcal {P}} $ and ${\mathcal {P}}(n) $ denote the set of partitions and the set of partitions of $n$. 
For a partition $\lambda$, we let $m_i(\lambda )$ denote the multiplicity of $i$ as its part. 
$(1^{m_1(\lambda )}2^{m_2(\lambda )} \ldots  )$ is another representation of $\lambda $. 
We define addition and subtraction of partition each two pattern. 
\[
\lambda +\mu :=(\lambda _1+\mu_1, \lambda _2+\mu_2 , \ldots, ), 
\]
\[
\lambda -\mu :=(\lambda _1-\mu_1, \lambda _2-\mu_2 , \ldots, ), 
\]
\[
\lambda \sqcup \mu :=(1^{m_1(\lambda )+m_1(\mu)}2^{m_2(\lambda )+m_2(\mu)}\ldots ),
\]
\[
\lambda \setminus \mu :=(1^{m_1(\lambda )-m_1(\mu)}2^{m_2(\lambda )-m_2(\mu)}\ldots ). 
\]
\indent
A partition $\lambda =(1^{m_1}2^{m_2}3^{m_3}\ldots ) $ is said to be odd if $m_{2i}=0$ for all $i$. 
We denote $\mathcal{OP}$ and $\mathcal{OP}(n) $ the set of odd partitions and the set of odd partitions of $n$.

For positive integer $N$, 
partition $\mu$ is called lecture hall partition when $\mu$ satisfies
\[
\frac{\mu_1}{N}\geq \frac{\mu_2}{N-1}\geq \ldots \geq \frac{\mu_{\ell(\mu)}}{N-\ell(\mu)+1}. 
\]
We call $N$ the width of lecture hall. 
We denote $\mathcal{L}_N$ the set of lecture hall partitions of width $N$.

\noindent
\textbf{Example. }
$(4, 3, 1) \not \in \mathcal{L}_3$ because $\frac{4}{3}>\frac{3}{1}$. 
$(4, 3, 1)\in \mathcal{L}_4$ because $ \frac{4}{4}\leq \frac{3}{3}\leq \frac{1}{2}\leq \frac{0}{1}$.

\bigskip
A lecture hall partition $\mu = (\mu_1, \mu_2, \ldots , \mu_{\ell(\mu)})$ is said to be reduced if  
\[
{}^\forall i, (\mu_1, \mu_2, \ldots (\mu _i-(N-i+1)), \ldots , \mu_{\ell_{\mu}} )\not \in \mathcal{L}_N . 
\]
We let $\mathcal{RL}_N$ denote 
the set of reduced lecture hall partitions of width $N$.\\
\noindent
\textbf{Example. }
For $N=3$. 
\[
\mathcal{RL}_3=\{ \emptyset , (1), (2), (2,1), (3,1), (4,1))\}. 
\]
Here $(5,1)\in \mathcal{L}_N$, but $(5,1)\not \in \mathcal{RL}_N$. 
Because $\frac{5}{3}-1>\frac{1}{2}$.   

\bigskip
For positive integer $k, N, (k\leq N)$, we denote that, 
\begin{eqnarray*}
[\Diamond]_{N,k}:=\left \{ 
\begin{array}{ll}
(N, N-1, \ldots , N-2k+2) &\ \ ({}^\forall k\leq \frac{N}{2})\\
(N, N-1, \cdots , 2k-N-1) &\ \ ({}^\forall k> \frac{N}{2})
\end{array}
\right.
. 
\end{eqnarray*}
For any $[\Diamond]_{N, k}=\mu=(\mu_1, \mu_2, \ldots , \mu_\ell)$, obviously 
\[
\frac{\mu_1}{N}=\frac{\mu_2}{N-1}=\ldots =\frac{\mu_\ell}{N-\ell +1}=1
\]
Then $[\Diamond]_{N,k}$ is lecture hall partition of width $N$ but not reduced. 
\begin{prop}\label{prop;1}
For any positive integer $N$, 
\[
\sharp \mathcal{RL}_N =N!. 
\]
\end{prop}
\begin{proof}
By the definition, 
\[
\mu =(\mu_1, \mu_2, \ldots , \mu_{\ell (\mu)} ) \in \mathcal{RL}_N 
\Rightarrow (\mu_2, \ldots , \mu_{\ell(\mu)}) \in \mathcal{RL}_{N-1} . 
\]
When $\mu_2,  \ldots , \mu _{\ell(\mu)}$ are fixed, 
there are $N$ pattern $m_1$ such that $\mu \in \mathcal{RL}_N$. 
Then it was proved by mathematical induction. 
\end{proof}
We also ready reduced odd partition. 
Let $N$ be a positive integer. 
We denote, 
\[
\mathcal{OP}_N:=\{ \lambda \in \mathcal{OP} \ |\ {}^\forall i> 2N, m_i=0\} , 
\]
\[
R\mathcal{OP}_N:=\{ \lambda \in \mathcal{OP}_N \ |\ {}^\forall k\leq N, m_{2k-1}\leq N-k \} . 
\]
We call them the set of $N$-party odd partitions, the set of $N$-party reduced odd partitions. 
The next proposition is obvious from definition. 
\begin{prop}\label{prop;2}
For any positive integer $N$, 
\[
\sharp \mathcal{OP}_N =N!. 
\]
\end{prop}
Then $\sharp \mathcal{RL}_N =\sharp \mathcal{ROP}_N$. 
We also prove the orders of these sets are equal in every size $n$. 
\begin{thm}[Reduced Lecture Hall Partition Theorem]\label{thm;rl1}
For any positive integer $N$, 
\[
\sum_{\mu \in \mathcal{RL}_N}{q^{|\mu |}}=\sum_{\lambda \in R\mathcal{OP}_N}{q^{|\lambda |}}. 
\]
\end{thm}
For proof of this theorem, we construct the map from $\mathcal{OP}_N$ to $\mathcal{L}_N$
\[
\begin{array}{ccccc}
\Phi _N: &\mathcal{OP}_N &\longrightarrow &\mathcal{L}_N&\\
& \rotatebox{90}{$\in$} & & \rotatebox{90}{$\in$} &\\
&\lambda &\longmapsto &\mu &\textrm{s.t.} |\lambda |=|\mu|
\end{array}. 
\]
And we prove this map is bijection.
We remark beforehand the map $\Phi _N$ is equivalent with \cite{Y}. 
However, the using sets become finite by reduction. 
Then the proof becomes easy. 
And we also give bijective proof of lecture hall partition theorem by $\Phi _N$.

\bigskip
Fist, let consider the case $N=\infty$, $\Phi:=\Phi_\infty$. 
For $\lambda =((2k-1)^m)$ which consists of a kind of part $2k-1$, 
we define
\[
\Phi ((2k-1)^m):= [\Diamond]_{k+m-1,k}. 
\]
And we put $A_{i,k}$  the increase from $\Phi ((2k-1)^{i-1})$ to $\Phi((2k-1)^{i})$. 
\footnote
{
$A_{i,k}$ may not be a partition. 
But we applied same calculation ``$+$", ``$-$". 
}
\[
A_{i,k}:=\Phi ((2k-1)^i)-\Phi ((2k-1)^{i-1})=[\Diamond]_{k+i-1,k}-[\Diamond]_{k+i-2,k}. 
\]
\textbf{Example. } Let $2k-1=5$, 
\[
\Phi(5^0) =[\Diamond]_{2,3}=\emptyset, 
\Phi(5) =[\Diamond]_{3,3} =(3,2), 
\Phi(5^2) =[\Diamond]_{4,3} =(4,3,2,1), 
\]
\[
\Phi(5^3) =[\Diamond]_{5,3} =(5,4,3,2,1), 
\Phi(5^4) =[\Diamond]_{6,3} =(6,5,4,3,2), \ldots . 
\]
Then 
\[
A_{1,3}=(3,2), A_{2,3}=(1,1,2,1), A_{3,3}=(1,1,1,1,1), A_{4,3}=(1,1,1,1,1), \ldots . 
\]
We grow up lecture hall partition by combination of this increase $A$. 
We define next counter $I(\lambda )=[I_1(\lambda ), I_2(\lambda ), \ldots , I_N(\lambda )]$ to select increase $A$. 
The default of counter is $I(\lambda )=[0, 0, \ldots 0]$. 
For ${}^\forall \lambda \in \mathcal{OP}_N, {}^\forall 2k-1\leq \lambda _{\ell(\lambda )}$, 
\[
\Phi _N(\lambda \sqcup (2k-1)^1):=\Phi _N(\lambda )+A_{i(\lambda ),k}. 
\]
Here, $i(\lambda ):=\textrm{min} \{ j\ | \ I_j(\lambda )=0 \}$. 
\begin{eqnarray*}
&&I(\lambda \sqcup (2k-1)^1)\\
&=&[I_1(\lambda )-1, \ldots, I_{i(\lambda )-1}(\lambda )-1, N-k-i(\lambda )+1, I_{i(\lambda )+1}(\lambda ), \ldots , I_N(\lambda )]. 
\end{eqnarray*}
\noindent
 \textbf{Example. }For $N=7$, $1^43^27^3911\in \mathcal{OP}_7 $. 
\begin{eqnarray*}
 \begin{array}{ll}
 \Phi (11)=\emptyset +A_{1,6}=(6,5),& I(11)=[1,0,0,0,0,0,0], \\ 
 \Phi (911)=(6,5)+A_{2,5}=(7,6,4,3),& I(911)=[0,1,0,0,0,0,0], \\
 \Phi (7911)=(7,6,4,3)+A_{1,4}=(11,9,4,3),& I(7911)=[3,1,0,0,0,0,0], \\
 \Phi (7^2911)=(11,9,4,3)+A_{3,4}=(12,10,5,4,2,1), &I(7^2911)=[2,0,1,0,0,0,0],\\
 \vdots ,& \vdots \\
\Phi (1^43^27^3911)=(20,13,9,6,2,1), &I(1^43^27^3911)=[1,3,2,0,1,1,0]. 
 \end{array}
 \end{eqnarray*}

\indent
We research the properties of counter $I$. 
For $\lambda \in \mathcal{OP} _N , \mu=\Phi_N (\lambda )$, 
We put $d_j(\lambda )$ the number of action $j$ while we grow up $\mu$.

\noindent
\textbf{Example. }
For $\lambda = 1^43^27^3911 \in \mathcal{OP} _7$ just before example, 
\[
d_1(\lambda )=3, d_2(\lambda )=3, d_3(\lambda )=2.
\]

\bigskip 
\begin{lem}\label{lem;2}
Let $\lambda \in \mathcal{OP} _N , \mu =\Phi_N (\lambda )$, $2k-1=\lambda _{\ell(\lambda )}$. 
Then, 
\begin{eqnarray}
{}^\forall j \leq k, \ \ \ &\mu_{2j-1}= (N-2j+2)d_j(\lambda )-I_j(\lambda )\\
{}^\forall j < k, \ \ \ &\mu_{2j}=(N-2j+1)d_j(\lambda )-I_j(\lambda ). 
\end{eqnarray}
Especially, 
\begin{eqnarray}
{}^\forall j \leq k, \ \ \ & \mu_{2j-1}\equiv -I_j(\lambda ) &({\rm{mod}}\ N-2j+2)\\
{}^\forall j < k, \ \ \ & \mu_{2j}\equiv -I_j(\lambda ) &({\rm{mod}}\ N-2j+1). 
\end{eqnarray}
\end{lem}
\begin{proof}
When default $\lambda =\mu =\emptyset, I(\emptyset )=[0,0,\ldots 0]$, 
both side equal $0$. 
\footnote
{
When $\lambda =\emptyset$, we regard $k$ as $2N-1$ (unlimited).  
}

Let prove equation inductive. 
We put $\lambda '=\lambda \setminus (2k-1), \mu'=\Phi_N (\lambda ')$.
And we assume that $\lambda '$ satisfies identity (3.1), (3.2). 
If $i(\lambda ')>j$ , 
\[d_j(\lambda)=d_j(\lambda '), I_j(\lambda )=I_j(\lambda ')-1.
\] 
From definition of $A$, 
\begin{eqnarray*}
{}^\forall j \leq k, \ \ \ &\mu_{2j-1}=\mu'_{2j-1}+1, \\
{}^\forall j < k, \ \ \ &\mu_{2j}=\mu'_{2j}+1. 
\end{eqnarray*}
If $i(\lambda ')=j$, 
\[
d_j(\lambda)=d_j(\lambda ')+1,  
I_j(\lambda )=I_j(\lambda ')+N-k-j+1, 
\] 
\[
{}^\forall j \leq k, \ \ \ \mu_{2j-1}=\mu'_{2j-1}+k-j+1, \mu_{2j}=\mu'_{2j}+k-j. 
\]
If $i(\lambda ')<j$, 
\[
d_j(\lambda )=d_j(\lambda '), I_j(\lambda )=I_j(\lambda '), \mu_{2j-1}=\mu'_{2j-1}, \mu_{2j}=\mu'_{2j}. 
\]
Then $\lambda $ satisfies equations. 
\end{proof}
\begin{lem}\label{lem;3}
Let $\lambda \in \mathcal{ROP} _N , (I, \mu) =\Phi (\lambda )$. 
For all $j$, 
\[
0\leq d_j(\lambda )-d_{j+1}(\lambda )\leq 1
\]
\end{lem}
\begin{proof}
By the definition of $I$ and $\Phi _N$, 
first $A_{j, \bullet }$ is previous than first $A_{j+1, \bullet }$. 
When $d_j(\lambda )-d_{j+1(\lambda )}=1$, 
$I_j(\lambda )> I_{j+1}(\lambda )$. 
\footnote
{
When $\lambda \not \in \mathcal{ROP}_N$, there are case $I_j(\lambda )=I_{j+1}(\lambda)=0$. 
}
Then, next $A_{j+1, \bullet }$ is previous than $A_{j, \bullet }$. 
When $d_j(\lambda )-d_{j+1(\lambda )}=0$, 
$I_j(\lambda )\leq I_{j+1}(\lambda )$. 
Then, next $A_{j, \bullet }$ is previous than $A_{j+1, \bullet }$. 
\end{proof}
Let $\lambda \in \mathcal{OP}_N$, $\mu=\Phi _N(\lambda ) \in \mathcal{RL}_N$, $\lambda _{\ell}(\lambda )=2k-1$. 
For $l \leq k$, we put $\lambda ':=\lambda \sqcup (2l-1)$, $\mu ':=\Phi _N(\lambda ')$. 
Because $\ell(A_{\bullet ,l})\leq 2l -1$, 
$\mu_j$ equals $\mu'_j$ for all $j$ greater than $2l-1$. 
Then, 
\[
(\mu'_{2l } , \mu'_{2l +1}, \ldots , \mu'_{\ell(\mu')}) \in \mathcal {RL}_{N-(2l -1) }. 
\]
Therefore the possibility of failure of
the inequalities of ``reduced" and ``lecture hall" exists in only $\mu'_j$s $(j\leq 2l )$. 
By the Lemma \ref{lem;3} and  (3.1), (3.2) of Lemma \ref{lem;4}, 
\[
\mu '\not \in \mathcal{L}_N
\Leftrightarrow {}^\exists j<l , d_j(\lambda' ) -d_{j+1}(\lambda' )=0 \wedge I_j(\lambda ')  >I_{j+1}(\lambda '). 
\]
The right-hand side is false by the argument of the proof of Lemma \ref{lem;3}. 
Then $\Phi _N$ is map from $\mathcal{OP}_N$ to $\mathcal{L}_N$. 
And, 
\[
\mu '\not \in \mathcal{RL}_N
\Leftrightarrow {}^\exists j<l , d_j(\lambda ') -d_{j+1}(\lambda ')=1 \wedge I_j(\lambda ') \leq I_{j+1}(\lambda '). 
\]
The right-hand side is false when $\lambda '\in \mathcal{ROP}_N$. 
It follows the next proposition. 
\begin{prop}\label{prop;3}
Let $\lambda \in \mathcal{OP}_N$. Then, 
\[
\lambda \in \mathcal{ROP}_N \Longrightarrow \Phi _N(\lambda ) \in \mathcal{RL}_N. 
\]
\end{prop}
\begin{lem}\label{lem;4}
Let $\lambda \in \mathcal{OP}_N, 2k-1 \leq \lambda _{\ell(\lambda )}$. Then, 
\[
\Phi _N (\lambda \sqcup (2k-1)^{N-k+1})=\Phi _N(\lambda )+[\Diamond ]_{N,k}, 
\]
\[
I(\lambda \sqcup (2k-1)^{N-k+1}) =I(\lambda ). 
\] 
\end{lem}
\begin{proof}
For any $a\leq N-k+1$, we put $\lambda ^{(a)}:= \lambda \sqcup (2k-1)^a$. 
Because the smallest part of $\lambda ^{(a)} $ is $2k-1$, 
\[
I_{N-k+1}(\lambda ^{(a)})=I_{N-k+2}(\lambda ^{(a)})=\ldots =I_N(\lambda ^{(a)})=0. 
\]
Then $i(\lambda ^{(a)})\leq N-k$. 
We prove that 
\[
\{ i(\lambda ), i(\lambda ^{(1)}), \ldots , i(\lambda ^{(N-k)})\} = \{ 1,2, \ldots , N-k+1\} .
\] 
Then, 
\[
\Phi _N(\lambda \sqcup (2k-1)^{N-k+1})=\Phi _N(\lambda )+\sum_{a=0}^{N-k}{A_{i(\lambda ^{a}), k}}
=\sum_{j=1}^{N-k+1}{A_{j,k}}
=\Phi _N(\lambda )+[\Diamond ]_{N,k}. 
\]
\indent
First ,we assume that $i(\lambda ^{(a)})$ is not equal to $1$ for all $a$. 
Because $I_1(\lambda )$ is less than $N-k$ and $I_1(\lambda ^{(a+1)})= I_1(\lambda ^{(a)})-1$, 
$I_1(\lambda ^{(N-k+1)})<0$. 
It is incompatible. 
We put $i(\lambda ^{(a_1)})=1$. 
Then $I_1(\lambda ^{(a_1+1)})$ equals $N-k$. 
For all $a$ bigger than $a_1$, $I_1(\lambda ^{(a)}) >0$. 
Therefore the $a$ that $i(\lambda ^{(a)})=1$ is only $a_1$.  \\
\indent
Next, we assume that $i(\lambda ^{(a)})$ is not equal to $2$ for all $a$. 
Because $I_2(\lambda )$ is less than $N-k-1$ 
and $I_2(\lambda ^{(a+1)})=I_2(\lambda ^{(a)})-1$ for all $a\not =a_1$, 
$I_2(\lambda ^{(N-k+1)})<0$. 
It is compatible too. 
We put $i(\lambda ^{(a_2)})=2$. 
Then $I_2(\lambda ^{(a_1+1)})$ equals $N-k-1$. 
For all $a$ bigger than $a_2$, $I_2(\lambda ^{(a)}) >0$. 
Therefore the $a$ that $i(\lambda ^{(a)})=2$ is only $a_2$.  
\[
\vdots
\]
For all $j$ less than $N-k+1$, 
the $a$ that $i(\lambda ^{(a)})=j$ is only $a_j$.
And we recall that $i(\lambda ^{(a)})\leq N-k+1$. 
Then the last one action is $A_{N-k+1,k}$.

Let fix $l\leq N-k+1$. 
\begin{eqnarray*}
I_l(\lambda ^{(a_j+1)})=
\left \{
\begin{array}{ll}
I_l(\lambda ^{(a_j)}) &(j<l)\\
I_l(\lambda ^{(a_j)})+N-k-l+1 &(j=l)\\
I_l(\lambda ^{(a_j)})-1 &(j>l)
\end{array}
\right .
.
\end{eqnarray*}
Then, 
\[
I_l(\lambda ^{(N-k+1)})=I_l(\lambda )+(N-k-l+1)-1\times (N-k-l+1)=I_l(\lambda ). 
\]
\end{proof}
\begin{prop}\label{prop;4}
Let $\lambda \in \mathcal{OP}_N$. Then, 
\[
\Phi_N (\lambda \sqcup (2k-1)^{N-k+1})=\Phi_N (\lambda )+[\Diamond ]_{N,k}. 
\]
Therefore, 
\[
\lambda \not \in \mathcal{ROP}_N\Longrightarrow \Phi _N(\lambda ) \not \in \mathcal{RL}_N. 
\]
\end{prop}
\begin{proof}
By the Lemma \ref{lem;4}, 
\[
2k-1 \leq \lambda _{\ell(\lambda )}
\Rightarrow 
I(\lambda \sqcup (2k-1)^{N-k+1}) =I(\lambda ). 
\] 
Then the growths after that are not change. 
\end{proof}
\begin{lem}\label{lem;5}
Let $\lambda , \mu \in R\mathcal {OP}_N $, $\Phi_N(\lambda )= \Phi _N(\mu )$, 
$2k-1=(\lambda \sqcup \mu )_{\ell (\lambda \sqcup \mu)}$. 
Then, 
\[
{}^\forall l\leq k , 
\Phi _N(\lambda \sqcup (2l-1))=\Phi _N(\mu \sqcup (2l-1)).
\] 
\end{lem}
\begin{proof}
By 
$\Phi_N(\lambda )= \Phi _N(\mu )$ and (3.1), 
\[
{}^\forall j\leq k, I_j(\lambda )=I_j(\mu ). 
\] 
Then, 
\[
i(\lambda )\leq k \vee i(\mu) \leq k \Rightarrow i(\lambda )=i(\mu) \Rightarrow A_{i(\lambda ),l}=A_{i(\mu), l}. 
\]
From definition of $A$, 
\[
i(\lambda )> k \wedge  i(\mu) > k \Rightarrow A_{i(\lambda ),l}=A_{i(\mu), l}. 
\]
Then, 
\[
\Phi _N(\lambda \sqcup (2l-1))=\Phi _N(\mu \sqcup (2l-1))=\Phi _N(\lambda )+A_{i(\lambda ),l}. 
\] 
\end{proof}
\begin{prop}\label{prop;5}
$\Phi _N$ is injection. 
\end{prop}
\begin{proof}
Let $\lambda , \mu  \in \mathcal{OP}_N, \lambda \not = \mu$.
We assume that $\Phi _N(\lambda )=\Phi_N(\mu)$. 
And we put $2k-1=(\lambda \sqcup \mu )_{\ell (\lambda \sqcup \mu)}$. 
If $m_{2k-1}(\lambda )=m_{2k-1}(\mu )=:m_{2k-1}$, 
we transform $\lambda, \mu$ as
\begin{eqnarray*}
\lambda \mapsto \lambda ':= \lambda \setminus (2k-1)^{m_{2k-1}}, \\ 
\mu \mapsto \mu ':= \mu \setminus (2k-1)^{m_{2k-1}}.
\end{eqnarray*}
Then, 
\begin{eqnarray*}
\Phi _N(\lambda ')&=&\Phi _N(\lambda \sqcup (2k-1 )^{N-k-m_{2k-1}+1})-[\Diamond ]_{N,k}\\
&=&\Phi _N(\mu \sqcup (2k-1 )^{N-k-m_{2k-1}+1})-[\Diamond]_{N,k}=\Phi _N(\mu '). 
\end{eqnarray*}
Repeat this transform until the multiples of the smallest parts will be different and less than $2N-k$. 
Let $m_{2k-1}(\lambda )>m_{2k-1}(\mu)$. 
Then, 
\[
\lambda \sqcup (2k-1)^{N-k-m_{2k-1}(\lambda )} \not \in \mathcal{ROP}_N, 
\mu \sqcup (2k-1)^{N-k-m_{2k-1}(\lambda )} \in \mathcal{ROP}_N. 
\]
Therefore, 
\[
\Phi _N(\lambda \sqcup (2k-1)^{N-k-m_{2k-1}(\lambda )}) \not \in \mathcal{RL}_N, 
\Phi _N(\mu \sqcup (2k-1)^{N-k-m_{2k-1}(\lambda )}) \in \mathcal{RL}_N. 
\]
It is compatible. 
\end{proof}
Then $\Phi $ is injection between same order sets. 
Therefore $\Phi $ is bijection. 
We proved Theorem \ref{thm;rl1}.

\indent
For any positive integer $N$, 
\[
\sum_{\lambda \in \mathcal{ROP}_N}{q^{|\lambda |}}
=\sum_{\mu \in \mathcal{RL}_N}{q^{|\mu |}}
=\prod_{k=1}^{N}\frac{1-q^{\Diamond _{N,k}}}{1-q^{2k-1}}. 
\]
\section{Lecture hall partition theorem}
We consider reduction of odd partition and lecture hall partition. 
First, for odd partition. 
Let $\lambda=(1^{m_1} 2^{m_2} \ldots ) \in \mathcal{OP}_N$. 
If $m_k>N-k$, we transform $\lambda $ as
\[
(1^{m_1} 2^{m_2} \ldots k^{m_k} \ldots ) \mapsto (1^{m_1} 2^{m_2} \ldots k^{m_k-(N-k+1)} \ldots ). 
\]
Repeat this transform until $m_k\leq N-k$ for all $k$. 
Then the resulting partition $[\lambda ]$ is reduced odd partition.

Next, for lecture hall partition. 
Let $\mu =(\mu_1, \mu_2, \ldots ) \in \mathcal{L}_N$. 
If $\frac{\mu_j}{N-j+1}-1\geq \frac{\mu_{j+1}}{N-j}$, 
then we transform $\mu$ as 
\[
\mu \mapsto \mu -(N, N-1, \ldots , N-j+1) = 
\left \{
\begin{array}{ll}
\mu - [\Diamond]_{N, j}&(j:\textrm{odd})\\
\mu - [\Diamond]_{N, 2N-j+1}&(j:\textrm{even})
\end{array}
\right.
.
\]
Repeat this transform until $\frac{\mu_j}{N-j+1}-1< \frac{\mu_{j+1}}{N-j}$ for all $j$. 
Then the resulting partition $[\mu ]$ is reduced lecture hall partition.

By the Lemma \ref{lem;1}, $|[\Diamond ]_{N,k}|=|(2k-1)^{N-k+1}|=\Diamond_{N,k}$. 
Then, 
\begin{eqnarray*}
&&\sum_{\lambda \in \mathcal{OP}_N}{q^{|\lambda |}}\\
&=&\prod_{k= 1}^{N}{\sum_{j\geq 0}{q^{j\times |(2k-1)^{N-k+1}|}}} \times \sum_{\lambda \in \mathcal{ROP}_N}{q^{|\lambda |}}\\
&=&\prod_{k= 1}^{N}\frac{1}{1-{q^{\Diamond_{N,k}}}} \times \prod_{k= 1}^{N}{\frac{1-q^{\Diamond_{N-k}}}{1-q^{2k-1}}}\\
&=&\prod_{k= 1}^{N}{\sum_{j\geq 0}{q^{j\times |[\Diamond ]_{N,k}|}}} \times \sum_{\mu \in \mathcal{RL}_N}{q^{|\mu |}}\\
&=&\sum_{\mu \in \mathcal{L}_N}{q^{|\mu|}}. 
\end{eqnarray*}
\begin{thm}[Lecture Hall Partition Theorem \cite{b-me}]
For any positive integer $N$, 
\[
\sum_{\lambda \in \mathcal{OP}_N}{q^{|\lambda |}}
=\sum_{\mu \in \mathcal{L}_N}{q^{|\mu|}}
=\prod_{k=1}^{N}{\frac{1}{1-q^{2k-1}}}. 
\]
\end{thm}
Same correspondence of $(2k-1)^{N-k+1}$ and $[\Diamond]_{N,k}$ was proved 
about a property of $\Phi_N$. 
Then $\Phi _N$ is also bijection from $\mathcal{OP}_N$ to $\mathcal{L}_N$.

\bigskip
Last of this paper, 
we introduce refined versions of lecture hall partition theorems. 
Let $\lambda =(\lambda _1, \lambda _2, \ldots )\in \mathcal{P}$, 
we denote the alternative size of $\lambda $
\[
|\lambda |_a:=\sum_{k=1}^{\ell(\lambda )}{(-1)^{k+1}\lambda _k}. 
\]
For trapezoid, 
\[
|[\Diamond ]_{N,k}|_a=k. 
\]
Then, 
\[
|A_{N,k}|_a=|[\Diamond]_{N,k}|_a-|[\Diamond]_{N,k-1}|_a=1. 
\]
Next theorems follow from the definition of $\Phi _N$. 
\begin{thm}For any positive integer $N$, 
\[
\sum_{\lambda \in \mathcal{OP}_N}{t^{\ell(\lambda )}q^{|\lambda |}}
=\sum_{\mu \in \mathcal{L}_N}{t^{|\mu |_a}q^{|\mu |}}
=\prod_{k=1}^{N}\frac{1}{1-tq^{2k-1}}
.
\]
\end{thm}
\begin{thm}For any positive integer $N$, 
\[
\sum_{\lambda \in \mathcal{ROP}_N}{t^{\ell(\lambda )}q^{|\lambda |}}
=\sum_{\mu \in \mathcal{RL}_N}{t^{|\mu |_a}q^{|\mu |}}
=\prod_{k=1}^{N}{\frac{1-t^{N-k+1}q^{\Diamond_{N,k}}}{1-tq^{2k-1}}}
.
\]
\end{thm}


\begin{thebibliography}{7}
\bibitem{b-me}
M. Bousquet-M$\acute{\rm{e}}$lou and K. Eriksson, 
Lecture hall partitions, 
{\it Ramanujan J.} {\bf{1}} (1997a) 101-110.

\bibitem{Y}
A. J. Yee, On the combinatorics of lecture hall partitions, 
{\it Ramanujan J.} {\bf{5}} (2001) 247-262. 


\end{thebibliography}
\end{document}